\documentclass[11pt]{amsart}
\usepackage{amssymb}%
\usepackage{graphicx}%
\usepackage{amsthm}%
\usepackage{amsmath}%
\usepackage{amsfonts}
\usepackage{latexsym}%
\usepackage{epsfig}%
\usepackage{epic}%
\usepackage{amscd}%
\usepackage{color}%
\usepackage{wrapfig}
\usepackage{amssymb}

\usepackage[headings]{fullpage}
\usepackage[linktoc=all, pagebackref, hyperindex]{hyperref}%
\hypersetup{colorlinks,  citecolor=blue,  filecolor=blue,  linkcolor=red,  urlcolor=blue}
\theoremstyle{plain}
\newtheorem{theorem}{Theorem}[section]
\newtheorem{proposition}[theorem]{Proposition}
\newtheorem{lemma}[theorem]{Lemma}
\newtheorem{corollary}[theorem]{Corollary}

\theoremstyle{definition}
\newtheorem{definition}[theorem]{Definition}
\newtheorem{example}[theorem]{Example}
\newtheorem{examples}[theorem]{Examples}

\theoremstyle{remark}
\newtheorem{remark}[theorem]{Remark}
{\begin{list}
		{\noindent\makebox[0mm][r]{\rm(\arabic{enumi})}}
		{\leftmargin=6.3ex \usecounter{enumi}}
	}
	{\end{list}}
\def\opn#1#2{\def#1{\operatorname{#2}}} 
\opn\Ap{Ap}
\opn\ord{ord}
\opn\Max{Max}
\opn\gr{gr}
\opn{\PF}{PF}
\opn{\tr}{tr}
\newcommand{\du}{S \! \Join^b \! E}
\definecolor{mygreen}{cmyk}{0.7,0.2,1,0}

\def\ZZ{{\mathbb Z}}

\def\fm{{\mathfrak m}}
\def\fn{{\mathfrak n}}

\newcommand{\excise}[1]{}
\begin{document}
	
	\title{ Tangent cones of monomial curves obtained by numerical duplication  }
	
	\author{Marco D'Anna}
	\address{Marco D'Anna - Dipartimento di Matematica e Informatica - Universit\`a degli Studi di Catania - Viale Andrea Doria 6 - 95125 Catania - Italy}
	\email{mdanna@dmi.unict.it}
	
	\author{Raheleh Jafari}
	
\address{Raheleh Jafari - Mosaheb Institute of Mathematics, Kharazmi University, and 	
		School of Mathematics, Institute for
		Research in Fundamental Sciences (IPM), P. O. Box 19395-5746,
		Tehran, Iran. }

	\email{rjafari@ipm.ir}

	\author{Francesco Strazzanti}
	\address{Francesco Strazzanti - Institut de Matem\`{a}tica - Universitat de Barcelona - Gran Via de les Corts Catalanes 585 - 08007 Barcelona - Spain}
	\email{francesco.strazzanti@gmail.com}

\thanks{The second author was in part supported by a grant from IPM (No. 96130112)}

	\thanks{The third author was partially supported by INdAM, MTM2013-46231-P and MTM2016-75027-P (Ministerio de Economı\'ia y Competitividad), and FEDER}


		\begin{abstract}
		Given a numerical semigroup ring $R=k[\![S]\!]$, an ideal $E$ of $S$ and  an odd element $b \in S$, 
		the numerical duplication $\du$ is a numerical semigroup, whose associated ring $k[\![\du]\!]$
		shares many properties with the Nagata's idealization and the amalgamated duplication of $R$ 
		along the monomial ideal $I=(t^e \mid e\in E)$. In this paper we study the associated graded ring of the numerical duplication characterizing when it is Cohen-Macaulay, Gorenstein or complete intersection. We also study when it is a homogeneous numerical semigroup, a property that is related to the fact that a ring has the same Betti numbers of its associated graded ring. On the way we also characterize when $\gr_\fm(I)$ is Cohen-Macaulay and when $\gr_\fm(\omega_R)$ is a canonical module of $\gr_\fm(R)$ in terms of numerical semigroup's properties, where $\omega_R$ is a canonical module of $R$.
	\end{abstract}

	\date{\today}
	
	\maketitle
	

\section*{Introduction}	

Let $(R, \mathfrak{m})$ be a commutative local ring. The study of the properties of its associated graded ring
$\gr_{\mathfrak m}(R)$ in connection with the properties of the original ring $R$ is a very difficult and interesting problem
in local algebra and it has been studied from many points of view, in the general case or for particular kind of rings.
One motivation for this study is the fact that in the geometrical context $\gr_\fm(R)$ corresponds to the 
tangent cone of the variety associated with $R$. 

A class for which this problem has been deeply studied is the class of numerical semigroup rings, i.e. rings of the form $k[\![S]\!]=k[\![t^s \mid s \in S$]\!], where $k$ is a field, $t$ an indeterminate and $S\subseteq \mathbb N$ an additive submonoid of the non-negative integers, with finite complement in $\mathbb N$.
These rings can be viewed as the completion of the coordinate ring of a monomial curve at the 
origin (i.e. its singular point). This means that the associated graded ring corresponds to the tangent cone 
at the origin of this curve. 

The aim of this paper is to study the tangent cone of monomial curves defined by numerical duplication,
finding numerical condition to determine their properties. Numerical duplication is a semigroup construction introduced in \cite{DS};
starting with a numerical semigroup $S$, a semigroup ideal $E$ and an odd element $b$ of $S$, it allows to construct a new semigroup 
denoted by $\du$, whose properties depend on $E$ and are often independent of $b$.
For example, if $E$ is a canonical ideal of $S$, the semigroup $\du$ is always symmetric, i.e. $k[\![\du]\!]$ is always Gorenstein.

In \cite{BDS}, the numerical duplication is connected to a more general ring construction. Let $R$ be a commutative ring, $I$ be an ideal of $R$ and $u,v \in R$;
if $\mathcal R_+(I)=\bigoplus_{i\geq 0} I^iT^i \subset R[T]$ denotes the corresponding Rees algebra, we set $R(I)_{v,u} = \mathcal R_+(I)/((T^2+vT+u)\cap \mathcal R_+(I))$. In particular, in \cite{BDS} it is proved that if $R=k[\![S]\!]$, $I=(t^{e_1}, \dots, t^{e_r})$ and $u=t^b$,
then $R(I)_{0,-u} \cong k[\![\du]\!]$, where $E=\{e_1,\ldots,e_r\}+S$ is the ideal of $S$ generated by $e_1, \dots, e_r$.

One interesting fact about this family of rings is that, as particular cases, we obtain other constructions such as the
Nagata's idealization (also called trivial extension; take $u=v=0$) or the amalgamated duplication 
(see \cite{D} and \cite{DF}; take $v=-1,\ u=0$).
Moreover, many relevant properties, such as Cohen-Macaulayness and Gorensteinnes are shared by all the 
rings in the family, independently of the choice of $u$ and $v$, see also \cite{BDS2}. Hence, results obtained in a particular case (like the case of numerical semigroup rings) give information on other kind of rings. For instance, this construction allowed the authors of \cite{OST} to find infinitely many one-dimensional Gorenstein rings with decreasing Hilbert function, starting from the particular case of numerical duplication.

In this paper we give numerical conditions that characterize when the tangent cone of $k[\![\du]\!]$ is Cohen-Macaulay, Gorenstein or a complete
intersection and we will see that these properties depend only on $E$ and not on $b$. Moreover, as a byproduct, we describe when the canonical module of $\gr_{\mathfrak m}(k[\![S]\!])$ has the expected form, i.e. when it is the associated graded module of the canonical module of $k[\![S]\!]$.
Subsequently, we study the homogeneous property for $\du$; this property has been introduced in \cite{JZ} and it is strictly connected to the homogeneous type property, i.e. the fact that $k[\![S]\!]$ and $\gr_\fm(k[\![S]\!])$ share the same Betti numbers. More precisely, in \cite{JZ} it is proved that, if $S$ has at most three generators, then $S$ is of homogeneous type if and only if either $S$ is homogeneous and $\gr_{\mathfrak m}(k[\![S]\!])$ is Cohen-Macaulay or $\gr_{\mathfrak m}(k[\![S]\!])$ is a complete intersection.
In the case of numerical duplication we characterize the homogeneous property and show that it depends also on $b$. Moreover, using this fact, we are able to construct semigroups that are of homogeneous type but not homogeneous and  their associated graded ring is not complete intersection, giving an answer to \cite[Question 4.22]{JZ}.

\medskip

The structure of the paper is the following. In Section 1, we fix the notation and prove some results about the associated graded ring of $R(I)_{v,u}$, in particular, Corollary \ref{CM-Gor} relates  the Cohen-Macaulayness of $\gr_\fm(k[\![\du]\!])$ to the Cohen-Macaulayness of $\gr_\fm(k[\![S]\!])$ and $\gr_\fm(I)$. For this reason we focus the subsequent section on the Cohen-Macaulayness of $\gr_\fm(I)$, listing several equivalent conditions in Proposition \ref{ECM}. In Section 3, we characterize the Gorenstein and complete intersection properties of $\gr_\fm(k[\![\du]\!])$, see Theorem \ref{Gorenstein} and Proposition \ref{complete intersection}; moreover, in Corollary \ref{canonical module} we determine when the canonical module of $\gr_\fm(k[\![S]\!])$ has the expected form, provided that it is Cohen-Macaulay. In the last section we characterize when $\du$ is homogeneous and we construct numerical semigroups that are of homogeneous type but not homogeneous and their associated graded rings are not complete intersection, see Theorem \ref{homogeneous duplication} and Example \ref{example}.1.

Several computations are performed by using the GAP system \cite{GAP} and, in particular, the NumericalSgps package \cite{DGM}.


\section{Preliminaries, idealization and tangent cone of duplication}

A {\it numerical semigroup} $S$ is a submonoid of $(\mathbb N, +)$ such that $\mathbb{N} \setminus S$ is finite. It is well known that $S$ is finitely generated and has a unique minimal system of generators. Throughout the whole paper,   $S=\langle n_1, \dots, n_\nu \rangle$ is a numerical semigroup minimally generated by $n_1<\cdots<n_\nu$ and   $R=k[\![S]\!]=k[\![t^{n_1}, \dots, t^{n_{\nu}}]\!]$ is the corresponding numerical semigroup ring with  maximal ideal  $\mathfrak m=(t^{n_1}, \dots, t^{n_{\nu}})$. The smallest nonzero element of $S$, $n_1$, is called the {\it multiplicity } of $S$ and is denoted by $m$; it is well known that $m=e(R)$, the multiplicity of $R$. 
A {\it relative ideal} of $S$ is a non-empty set $E$ of integers such that $E+S\subseteq E$ and $s+E\subseteq S$ for some $s\in S$. A relative ideal of $S$ that  is contained in $S$ is called an {\it ideal} of $S$.  We always assume  that an ideal does not contain $0$, i.e. $E \neq S$. Note that for relative ideals $E_1$ and $E_2$  of $S$, the set $E_1+E_2=\{e_1+e_2 \mid e_1\in E_1 , e_2\in E_2\}$  is also a relative ideal. In particular, for $z\in \ZZ$, $z+S=\{z+s \mid s\in S\}$ is the principal relative ideal of $S$ generated by $z$. 
For any ideal $E$  of $S$, we can always express it as $E=(e_1+S) \cup \dots \cup (e_r+S)$, for some $e_i \in E$;
then, we will write $E=\{e_1, \dots, e_r\}+S$ and we can always assume that the set $\{e_1,\dots,e_r\}$ is minimal,
i.e., for all $i=1,\dots,r$, $e_i \notin \bigcup_{j\neq i}(e_j+S)$. It is straightforward to see that $E$ has a unique minimal set of generators.
By difference of two ideals $E_1$ and $E_2$, we mean the ideal $E_1-E_2=\{z\in\ZZ \mid z+E_2\subseteq E_1\}$. 
We denote by $M=S \setminus \{0\}$  the maximal ideal of $S$
and we set $lM=M+ \dots + M$.
The {\it blowup} of $S$ is defined as the numerical semigroup 
$$
S'= \bigcup_l \ (lM - lM) = \langle m, n_2 - m, \dots, n_\nu - m \rangle.
$$
It is well known that $S'=lM-lM=lM-lm$ for $l$ large enough (cf. \cite[Proposition~1.1]{L}).

Let $\omega_i=\min \{s \in S \mid  s \equiv i \ (\text{mod} \ m)\}$. The {\it Ap\'ery set} of $S$ with respect to $m$ is the set $\Ap_{m}(S)=\{\omega_0=0, \omega_1, \dots, \omega_{m-1}\}=\{s\in S \mid s-m\notin S\}$. In the same way we denote $\Ap_{m}(S')=\{\omega'_0=0, \omega'_1, \dots, \omega'_{m-1}\}$. It follows from the definition that $\omega_i \geq \omega'_i$ for all $i=0, \ldots, m-1$ and we define the {\it microinvariants} of $S$ as the integers $a_i(S)$ such that $\omega'_i + m \, a_i(S)=\omega_i$.
Moreover, we set $b_i(S)=\max\{l \mid \omega_i \in lM\}$.
A criterion for the Cohen-Macaulayness of the associated graded ring, proved by Barucci and Fr\"{o}berg in \cite[Theorem 2.6]{BF}, implies the following theorem:

\begin{theorem} \label{BaFr}
The graded ring $\gr_{\mathfrak m}(R)$ is Cohen-Macaulay if and only if $a_i(S)=b_i(S)$ for each $i=0, \ldots, m-1$.
\end{theorem}

Let $E=\{e_1, \dots, e_r\}+S$ be an ideal of $S$ and let $b \in S$ be an odd integer.
The {\it numerical duplication} of $S$ with respect to $E$ and $b$ is defined in \cite{DS} as the numerical semigroup
$$
\du = 2 \cdot S \cup \{2 \cdot E + b\},
$$ 
where $2 \cdot X=\{2x \mid x \in X \}$. It is easy to see that 
$$
\du=\langle 2n_1, \dots,2n_\nu, 2e_1+b, \dots,2e_r+b \rangle.
$$

As we noticed in the introduction, the numerical duplication can be connected to the construction of the rings
$R(I)_{v,u} = \mathcal R_+(I)/((T^2+vT+u)\cap \mathcal R_+(I))$. In fact, if $R=k[\![S]\!]$, $I=(t^{e_1},	\dots , t^{e_r})$
and $u=t^b$, then $R(I)_{0,-u} \cong k[\![\du]\!]$.
More precisely, every element of $R(I)_{v,u} $ can be uniquely written in the form $f+gT$ with the multiplication induced by the equation $T^2+vT+u=0$.
Thus, when $R=k[\![S]\!]$, $v=0$ and $u=t^b$ the multiplication in $R(I)_{0,-u}$ is given by
\[(f(t)+g(t)T)(h(t)+l(t)T)=f(t)h(t)+t^bg(t)l(t)+(f(t)l(t)+g(t)h(t))T.\]
Hence, if $I=(t^{e_1},	\dots , t^{e_r})$
is a monomial ideal, it is easy to check that the map $R(I)_{0,-u} \rightarrow k[\![\du]\!]$
given by $f(t)+g(t)T\mapsto f(t^2)+g(t^2)t^b$ is an isomorphism
(we remark that in \cite{BDS} the authors forgot to state the hypothesis that $I$ has to be a monomial ideal).  

Now we are interested in studying the associated graded ring of $k[\![\du]\!]$. To this aim we prove a more general result. We recall that the Nagata's idealization is defined as follows: let $A$ be a ring and $N$ be an $A$-module; then, $A \ltimes N$ is the $A$-module $A \oplus N$ with the multiplication defined as $(r,m)(s,n)=(rs,rn+sm)$.
If $(A,\mathfrak{n})$ is local, then  $A\ltimes N$ is local with maximal ideal $\overline{\mathfrak n}:=\mathfrak n \oplus N$
and it is well-known that $\gr_{\overline{\mathfrak n}}(A\ltimes N)\cong \gr_{\mathfrak n}(A)\ltimes \gr_{\mathfrak n}(N)(-1)$; in fact, the homogeneous elements of degree $1$ are the elements of $\mathfrak n/\mathfrak n^2 \oplus N/\mathfrak nN$. For an element $f\in A$ with $\mathfrak{n}$-adic order $d$, the residue class of $f$ in $\mathfrak{n}^d/\mathfrak{n}^{d+1}$ is called the {\it initial form} of $f$ and is denoted by $f^*$. We use the same notation $f^*$ to denote the image of $f$ in $\gr_\fn(A)$.

We also recall that, if $(A,\mathfrak{n})$ is local, all the rings $A(I)_{v,u}$ are local with maximal ideal
$\overline{\mathfrak n}$ isomorphic, as $A$-module, to $\mathfrak{n}\oplus I$ (\cite[Proposition 2.1]{BDS}). With this notation, we can state the following general result.

\begin{proposition}\label{gr general}
Let $(A,\mathfrak n)$ be a local ring and let $I$ be a proper ideal. Assume that both $u$ and $v$ belong to $\mathfrak n$.
Then $\gr_{\overline{\mathfrak n}}(A(I)_{v,u}) \cong \gr_{\mathfrak n}(A) \ltimes \gr_{\mathfrak n}(I)(-1)$. 
\end{proposition}

\begin{proof}
By the proof of \cite[Proposition 2.3]{BDS}, $\bar{\mathfrak n}^i/\bar{\mathfrak n}^{i+1}=\{ r^*+x^*T \mid r^* \in \mathfrak n^i/\mathfrak n^{i+1} \text{ \ and \ } x^* \in \mathfrak n^{i-1}I/\mathfrak n^{i}I \}$ for all $i >0$.
	Therefore, we get a bijective map $\varphi : \gr_{\bar{\mathfrak n}}(A(I)_{v,u}) \rightarrow \gr_{\mathfrak n}(A) \ltimes \gr_{\mathfrak n}(I)(-1)$ by setting $\varphi(r^*+x^*T)=(r^*,x^*)$ for every homogeneous element $r^*+x^*T$ of $\gr_{\bar{\mathfrak n}}(A(I)_{v,u})$. 
	This is a ring isomorphism, because the multiplication of two homogeneous elements of the first ring is induced by the multiplication in $A(I)_{v,u}$ and then, since $u,v \in \mathfrak n$, we have  
\begin{gather*}
	\varphi((r^*+ x^* T) ( s^*+ y^* T))=
	\varphi((rs-uxy)^*+(ry+sx-vxy)^*T)= 
	\varphi((rs)^*+(ry+sx)^*T)= \\
	=((rs)^*,(ry+sx)^*)=
	(r^*,x^*)(s^*,y^*)=\varphi(r^*+x^*T) \cdot \varphi(s^*+y^*T). \qedhere
	\end{gather*}  
\end{proof}

In  the rest of the paper,  $E$  is an ideal of $S$ not containing $0$, minimally generated by $\{ e_1,\dots, e_r \}$ and
$I=(t^{e_1},	\dots , t^{e_r})$ is the corresponding monomial ideal of $R=k[\![S]\!]$.

\medskip
The next corollary follows immediately from Proposition~\ref{gr general}.

\begin{corollary} \label{idealization}
Let  $D=k[\![\du]\!]$ and let $\mathfrak n$ be its maximal ideal. Then, $\gr_{\mathfrak n}(D)\cong \gr_{\mathfrak m}(R)\ltimes \gr_{\mathfrak m}(I)(-1)$.
\end{corollary}

Notice that we can give an explicit isomorphism between the two graded rings $\gr_\mathfrak{n}(D)$ and $\gr_{\mathfrak m}(R)\ltimes \gr_{\mathfrak m}(I)(-1)$,
since it is induced by the isomorphism $R(I)_{0,-t^b} \rightarrow k[\![\du]\!]$. More precisely, a
 degree $i$ homogeneus element  of $\gr_{\mathfrak m}(R)\ltimes \gr_{\mathfrak m}(I)(-1)$ is
in $\mathfrak m^i/\mathfrak m^{i+1} \oplus \mathfrak m^{i-1}I/\mathfrak m^{i}I$; so it is of the form 
$(x(t)^*,i(t)^*)$ and it corresponds to $\left(x(t^2)+i(t^2)t^b\right)^*\in \mathfrak{n}^i/\mathfrak{n}^{i+1}$.

From the previous corollary we can deduce the following result (see \cite[Corollary~4.14]{AW} and \cite[Theorem~5.6]{FGR}).

\begin{corollary}\label{CM-Gor} Let $D=k[\![\du]\!]$ and let $\mathfrak n$ be its maximal ideal. Then the following statements hold. 
	\begin{enumerate}
	\item The graded ring $\gr_{\mathfrak n}(D)$ is Cohen-Macaulay if and only if $\gr_{\mathfrak m}(R)$ is Cohen-Macaulay and $\gr_{\mathfrak m}(I)$ is a maximal Cohen-Macaulay module of $\gr_{\mathfrak m}(R)$; 
	\item $\gr_{\mathfrak n}(D)$ is Gorenstein if and only if $\gr_{\mathfrak m}(R)$ is Cohen-Macaulay and $\gr_{\mathfrak m}(I)$ is a canonical module of $\gr_{\mathfrak m}(R)$.
	\end{enumerate}
\end{corollary} 

Notice that if $\gr_{\mathfrak n}(D)$ is Gorenstein, also $D$ has to be Gorenstein and, therefore, $I$ has to be a canonical ideal for $R$. Nevertheless, it is not always true that the canonical module of $\gr_{\mathfrak m}(R)$ is of the expected form, i.e. of the form $\gr_{\mathfrak m}(\omega_R)$, with $\omega_R$ canonical module of $R$; see Corollary \ref{canonical module}.

Our aim in the next two sections is to find numerical conditions on $E$ so that $\gr_{\mathfrak n}(D)$ is Cohen-Macaulay, Gorenstein or a complete intersection.

\section{Cohen-Macaulay property}

We would like to combine Theorem \ref{BaFr} and Corollary \ref{CM-Gor} (1) to study when the associated graded ring of the numerical duplication is Cohen-Macaulay. To this purpose we first study the associated graded module of an ideal in order to establish a result for ideals analogous to Theorem \ref{BaFr}. 

We define the Ap\'ery set of the semigroup ideal $E$ with respect to the multiplicity $m$ as $\Ap_{m}(E)=\{\alpha_0, \dots, \alpha_{m-1}\}$, 
where $\alpha_i$ is the smallest element in $E$ that is congruent to $i$ modulo $m$; we notice that $m$ may not be in $E$. 
Moreover, when we write $E=\{e_1, \dots, e_r\}+S$, we assume that the generators are in increasing order, so that
$e_1=\min(E)$.

We define the  ideal $E'$ of $S'$ as 
$$
E'=\bigcup_{l\geq 1} \ \left(E+(l-1)M\right) - lM.
$$
Notice that $E'=(E+(l-1)M)-lM$  for $l$ large enough.

\begin{lemma}
Let $S$ be a numerical semigroup with maximal ideal $M$ and multiplicity $m$. 
Then $E'=\{e_1-m, \dots, e_r-m\}+S'$.
\end{lemma}

\begin{proof}
	The inclusion $(\supseteq)$ follows from the fact that each $x \in \{e_1-m, \dots, e_r-m\}+S'$ 
	is of the form $x=e_i-m+y$ with $i=1,\dots,r$ and $y\in S'$. 
	Hence, for $l$ big enough, $y+lM\subseteq lM$ and so $x+lM \in (e_i-m)+lM =e_i+(l-1)M \subseteq E+(l-1)M $, where the equality holds by \cite[Proposition I.2.1(b)]{BDF}.
	
	Conversely, take $x$ such that $x+lM \subseteq E+(l-1)M$, for some $l$; in particular, $x+lm =e_i+y$, for some $i=1,\dots,r$, and $y \in (l-1)M$. Hence, $x=(e_i-m)+(y-(l-1)m)\in  (e_i-m)+S' \subseteq \{e_1-m, \dots, e_r-m\}+S'$.
\end{proof}

Note that, if $m \in E$, then $0\in E'$ and so $E'=S'$. Let $\{\alpha'_{0}, \dots,\alpha'_{m-1}\}$ be the Ap\'ery set with respect to $m$ of $E'$. Let  $a_i(E)$ denote the unique integer such that $\alpha'_i+m\,a_i(E)=\alpha_i$.
We notice that $a_i(E)$ is indeed the largest number $\lambda$ such that $\alpha_i-\lambda m\in E'$, i.e the smallest number $\lambda$ such that $\alpha'_i+\lambda m\in E$.
We also define the order of $e \in E$ as the integer $$\ord_E(e):=\max\{l+1 \mid e\in lM+E\};$$ moreover, for
all $i=0, \dots, m-1$, we set $ b_i(E):=\ord_E( \alpha_i).$ We will use $a(E)$ and $b(E)$, respectively, to denote the vectors $[a_0(E),\ldots,a_{m-1}(E)]$ and $[b_0(E),\ldots,b_{m-1}(E)]$.

\begin{remark}\label{b less a}
	If $\alpha_i\in lM+E$, then  $\alpha_i=\sum^p_{j=1}r_js_j+e$, for some $s_j\in M$, $e\in E$ and $\sum^p_{j=1}r_j=l$. Therefore, $\alpha_i-(l+1)m=\sum^p_{j=1}r_j(s_j-m)+(e-m)\in E'$, since $e-m\in E'$. In particular, $a_i(E)\geq b_i(E)$.
\end{remark}

\begin{lemma}\label{ord E}
	Let $E$ be an ideal of $S$. The following statements are equivalent:
	\begin{enumerate}
		\item $\ord_E(e+m)=\ord_E(e)+1$, for all $e\in E$;
		\item $\ord_E(e+\lambda m)=\ord_E(e)+\lambda$ for all $e\in\Ap_{m}(E)$ and $\lambda \in \mathbb N$;
		\item $a_i(E)=b_i(E)$ for all $i=0, \ldots, m-1$.
	\end{enumerate} 
\end{lemma}

\begin{proof}
	(1)$\Rightarrow$(2). This is clear. 
	
	\smallskip
	(2)$\Rightarrow$(3). By Remark~\ref{b less a}, we have $a_i(E)\geq b_i(E)$ for all $i=0, \ldots, m-1$. 
	As $\alpha'_i\in E'$, by definition of $E'$ we get  $\alpha'_i+n m\in E+(n-1)M$, for $n\gg 0$. In particular, it follows that $\alpha_i+(n-a_i(E))m\in E+(n-1)M$ and, therefore, since $\ord_E(\alpha_i)=b_i(E)$, the hypothesis implies 
	\[b_i(E)+n-a_i(E)=\ord_E(\alpha_i+(n-a_i(E))m)\geq n.\]
	Hence, $b_i(E)\geq a_i(E)$ and the result follows.
	
	\smallskip
	(3)$\Rightarrow$(1). Let $e\in E$. Then $e=\alpha_i+\lambda m$ for some $\alpha_i\in\Ap_{m}(E)$ and $\lambda\geq 0$. If $\ord_E(e+m)>\ord_E(e)+1$, then  
	\[\ord_E(e+m)=\ord_E(\alpha_i+(\lambda+1)m)> b_i(E)+\lambda+1=a_i(E)+\lambda+1.\]
	
	Hence, $\alpha_i+(\lambda+1)m=\sum_j r_js_j +e$, where $\sum_j r_j=a_i(E)+\lambda+1$ and $e\in E$. Moreover, since $\alpha_i'=\alpha_i - a_i(E)m$, we get 
	\[\alpha'_{i}-m=\alpha_i + (\lambda+1)m- (a_i(E)+\lambda+2)m=\sum_j r_j(s_j-m)+(e -m)\in E',\]
	that is  contradiction, because $\alpha'_i \in \Ap_{m}(E')$.
\end{proof}


\begin{lemma}\label{Gl1} The following statements hold true.
	\begin{enumerate} 
		\item 	Given $a\in S$ and $b\in E$, we have $(t^a)^* \cdot (t^b)^*=0$ if and only if $\ord_E(a+b)>\ord_S(a)+\ord_E(b)$. 
		\item  For any $s\in S\setminus\{m\}$,    $(t^{s})^*$ is nilpotent. 
	\end{enumerate}
\end{lemma}
\begin{proof}
The first statement is clear by definition and the second one is the subject of \cite[Lemma 5]{G}.
\end{proof}

\begin{remark}
	We notice that the only minimal monomial prime ideal of $\gr_{\mathfrak m}(R)$ is the ideal generated by $\{(t^{n_2})^*,\ldots,(t^{n_\nu})^*\}$, cf.~\cite[Corollary~2.3]{Huang-2015}. Moreover, taking  $I=(t^{e_1},	\dots , t^{e_r})$, since $\fm^{n-1}I\neq\fm^nI$ for all $n$,  $\gr_{\mathfrak m}(I)$ has positive dimension and, then, has dimension one.
\end{remark}

\begin{proposition} \label{ECM} Let $E=\{e_1, \dots, e_r\}+S$ and let $I=(t^{e_1},	\dots , t^{e_r})$.
	The following statements are equivalent:
	\begin{enumerate}
		\item  $\gr_{\mathfrak m}(I)$ is a one-dimensional  Cohen-Macaulay $\gr_{\mathfrak m}(R)$-module;
		\item $(t^{m})^*$ is not a zero-divisor of $\gr_{\mathfrak m}(I)$;
		\item $\ord_E(e+m)=\ord_E(e)+1$, for all $e\in E$;
		\item $a_i(E)=b_i(E)$ for all $i=0, \ldots, m-1$.
	\end{enumerate}
\end{proposition}
\begin{proof}
	(1)$\Rightarrow$(2). Let $f=\sum^p_{i=1}k_i  (t^{s_i})^*$ be a non-zero-divisor of $\gr_{\mathfrak m}(I)$. 	If $k_j>0$ for some $s_j\neq m$, then  $(t^{s_j})^*$ is nilpotent by Lemma~\ref{Gl1}(2) and so $\sum_{i\neq j}k_i  (t^{s_i})^*$ is again a non-zero-divisor. Thus, we may assume that $s_i=m$ for all $i$, in particular $(t^{m})^*$ is a non-zero-divisor.
	
	(2)$\Rightarrow$(3). It follows directly by Lemma~\ref{Gl1}(1).
	
	(3)$\Rightarrow$(1).  We claim $(t^{m})^*$ is not a zero-divisor of $\gr_{\mathfrak m}(I)$ and the result follows immediately, since the dimension of $\gr_{\mathfrak m}(I)$ is one. Assume, on the contrary, that $(t^{m})^*\cdot f=0$ for some $f\in \fm^{n-1}I/\fm^nI$. We may write  $f=\sum^p_{i=1}h_i (t^{s_i})^*$, where $0\neq h_i\in k$ and  $s_1<\cdots<s_p$ belong to $(n-1)M+E\setminus nM+E$.  Now, 
	\[(t^{m+s_1})^*=\sum^p_{i=2}-\frac{h_i}{h_1}(t^{m+s_i})^* \in \mathfrak{m}^nI/\mathfrak{m}^{n+1}I.\]
Since $s_1<s_i$ for all $i=2,\ldots,p$, we get $(t^{m+s_1})^*=0  \in \mathfrak{m}^nI/\mathfrak{m}^{n+1}I$, that is equivalent to 
	\[s_1+m\in (n+1)M+E,\]
	i.e. $\ord_E(m+s_1)\geq n+2$, a contradiction.
	
	(3)$\Leftrightarrow$(4). It is proved in Lemma \ref{ord E}.
\end{proof} 

Since $\ord_S(s)=\ord_M(s)$ for every $s \in M$, we re-obtain as a particular case the following known result.

\begin{corollary}\label{CM2}\cite[Theorem~7, Remark~8]{G}, \cite[Theorem~2.6]{BF} 
	The following are equivalent: 
	\begin{enumerate}
		\item $\gr_{\mathfrak m}(R)$ is Cohen-Macaulay;
		\item $\ord_S(s+m)=\ord_S(s)+1$, for all $s\in M$;
		\item $\ord_S(s+\lambda m)=\ord_S(s)+\lambda$ for all $s\in\Ap_{m}(S)$ and $\lambda \in \mathbb N$;
		\item $a_i(S)=b_i(S)$, for all $i=0, \ldots, m-1$.
	\end{enumerate}
\end{corollary}

We are now ready to give a numerical interpretation of Corollary \ref{CM-Gor} (1). We recall that we denote by $(D, \mathfrak n)$ the ring $k[\![\du]\!]$. 

\begin{theorem}
 The associated graded ring   $\gr_{\mathfrak n}(D)$ is Cohen-Macaulay if and only if
	$a_i(S)=b_i(S)$ and $a_i(E)=b_i(E)$ for all $i=0, \ldots, m-1$.
\end{theorem}

\begin{proof}
	By Corollary \ref{CM-Gor} (1) we have to give necessary and sufficient conditions for $\gr_{\mathfrak m}(R)$ and $\gr_{\mathfrak m}(I)$ to be Cohen-Macaulay. The thesis follows immediately by Theorem \ref{BaFr} and  Proposition \ref{ECM}.
\end{proof}

Let us explore some consequences of Proposition \ref{ECM}.

\begin{corollary}
	Let $\gr_{\mathfrak m}(I)$ be a Cohen-Macaulay $\gr_{\mathfrak m}(R)$-module. If $m\in E$, the minimal set of generators of $E$ is a subset of the minimal generating set of $S$.
\end{corollary}
\begin{proof}
	Assume, on the contrary, that the minimal generating set of $E$ has an element $e=s_1+s_2$  for two positive elements $s_1,s_2\in S$. Then, $\ord_E(e+m)\geq 3 > \ord_E(e)+1$ and this contradicts Lemma \ref{ord E}.
\end{proof}

Let $e\in E$ be such that $\ord_E(e)=n$; by a {\it maximal representation} of $e$ we mean a representation of the form
$e=\sum_i r_is_i+x$ such that $\sum_i r_i=n-1$ and $x\in E$. 

\begin{remark}\label{max-expression}
	Let $e\in E$ and $e=\sum_i r_is_i+x$ be a  maximal representation of $e$. If $x=a+s$ for some $a\in E$ and $s\in S$, then we get another representation $e=(\sum_i r_is_i+s)+a$ that implies $s=0$, as $\ord_E(e)=\sum_i r_i+1$. In other words, $x$ belongs to the minimal set of generators of $E$.
\end{remark}

\begin{corollary}\label{principal}
	Let $E$ be a principal ideal of $S$. If $\gr_{\mathfrak m}(R)$ is Cohen-Macaulay, then $\gr_{\mathfrak m}(I)$ is a Cohen-Macaulay $\gr_{\mathfrak m}(R)$-module. 
		In particular, if $E$ is a principal ideal, $\gr_{\mathfrak n}(D)$ is Cohen-Macaulay if and only if $\gr_{\mathfrak m}(R)$ is Cohen-Macaulay.
\end{corollary}
\begin{proof}
	Let $E=e+S$. If $a=\sum_i r_i n_i+x$ is a maximal representation of $a\in E$, then $x=e$ by Remark \ref{max-expression}. As $\gr_{\mathfrak m}(R)$ is Cohen-Macaulay, $\ord_S(\sum_i r_i n_i+m)=\sum_i r_i+1$ by Corollary \ref{CM2}. Let $a+m=\sum_i r'_i n_i+e$ be a maximal representation of $a+m\in E$. Then $\sum_i r'_i n_i=\sum_i r_in_i+m$ and, in particular, $\ord_E(a+m)=\ord_S(\sum_i r_in_i+m)+1=\ord_S(\sum_i r_in_i)+1+1=\ord_E(a)+1$. Now, the  Proposition~\ref{ECM} implies that $\gr_{\fm}(I)$ is a Cohen-Macaulay $\gr_{\mathfrak m}(R)$-module. The last statement follows by Corollary~\ref{CM-Gor}.
\end{proof}	

\begin{examples} \rm
{\bf 1.}	Consider the numerical semigroup $S=\langle 3,4 \rangle = \{0,3,4,6 \rightarrow \}$, the ideal $E=\{3,8\}+S= \{3,6 \rightarrow \}$, and the integer $b=3$; then $\du= \langle 6,8,9,19 \rangle = \{0,6,8,9,12,14 \rightarrow \}$. It is not difficult to see that
	$a(S)=b(S)=[0,1,2]$, but $a(\du)=[0,2,1,1,2,2]$ and $b(\du)=[0,1,1,1,2,2]$. Then $\gr_{\mathfrak n}(D)$ is not Cohen-Macaulay, although $\gr_{\mathfrak m}(R)$ is. In fact $a(E)=[1,2,2]$ and $b(E)=[1,2,1]$\\
{\bf 2.}	Consider the numerical semigroup $S=\langle 5,14,17 \rangle$, the ideal $E=\{14,20,22\}+S$ and $b=17$.
	In this case we have $\Ap_5(S)=\{0,31,17,28,14\}$, $\Ap_5(E)=\{0,20,31,22,28\}$ and $\Ap_{10}(\du)=\{0,61,62,73,34,45,56,57,28,79\}$. Moreover, it is possible to see that
	$a(S)=b(S)=[0,2,1,2,1]$ and $a(\du)=b(\du)=[0,1,2,2,1,1,2,1,1,2]$; then, both $\gr_{\mathfrak m}(R)$ and $\gr_{\mathfrak n}(D)$ are Cohen-Macaulay.
\end{examples}

\section{Gorenstein property}

In this section we will need to list the elements of  the Ap\'ery sets of $S$ and $E$, with respect to the multiplicity, in increasing order: hence, we will denote them by $\Ap_m(S)=\{\delta_1=0 < \delta_2 < \dots < \delta_m \}$ and $\Ap_m(E)=\{\beta_1 < \beta_2 < \dots < \beta_m\}$. It is straightforward to see that $\Ap_{2m}(\du)=\{2 \delta_1, \dots, 2 \delta_m, 2 \beta_{1}+b, \dots, 2 \beta_{m}+b\}$.

We define a partial ordering $\leq_{M}$ on $\Ap_m(S)$ by setting $\delta_i \leq_{M} \delta_j$ if there exists $\delta_k \in \Ap_m(S)$ such that $\delta_i + \delta_k = \delta_j$ and $\ord_S(\delta_i) + \ord_S(\delta_k)= \ord_S (\delta_j)$.
A numerical semigroup $S$ is said to be $M$-pure if all the maximal elements of $\Ap_m(S)$ with respect to $\leq_{M}$ have the same order. 
In \cite[Theorem 3.14]{B} L. Bryant proves that $\gr_\fm(R)$ is Gorenstein if and only if it is Cohen-Macaulay and $S$ is $M$-pure and symmetric.

	\begin{remark}\label{Bryant} \rm
		If $S$ is $M$-pure, then $\gr_\fm(R)$ is Cohen-Macaulay if and only if $\ord_S(\delta +\lambda m)=\ord_S(\delta)+ \lambda$ for all $\lambda \in \mathbb{N}$ and all maximal elements $\delta \in \Ap_m(S)$ with respect to $\leq_M$. Clearly, one implication follows from Corollary \ref{CM2}.
		Conversely, if $\delta_i \in \Ap_m(S)$, $\delta_i <_M \delta$ with $\delta$ maximal and $\delta_k=\delta -\delta_i$, then for all $\lambda \in \mathbb{N}$ we have
		\begin{equation*}
		\begin{split}
		\ord_S(\delta_i+\lambda m) \geq \ord_S (\delta_i)+ \lambda& = 
		\ord_S(\delta)+\lambda - \ord_S(\delta_{k})\\&=\ord_S(\delta + \lambda m)- \ord_S(\delta_{k})\\&= 
		\ord_S(\delta_i + \lambda m+\delta_{k})- \ord_S(\delta_{k})\\
		& \geq \ord_S(\delta_i + \lambda m)+\ord_S(\delta_k)- \ord_S(\delta_k)=\ord_S(\delta_i+\lambda m).
		\end{split}
		\end{equation*}
		Therefore, the first inequality is an equality and the claim follows from Corollary \ref{CM2}. In particular, if $S$ is $M$-pure and symmetric, then $\delta_m$ is the only maximal element by \cite[Proposition 3.7]{B} and, thus, $\gr_\fm(R)$ is Cohen-Macaulay if and only if $\ord_S(\delta_m +\lambda m)=\ord_S(\delta_m)+ \lambda$ for all $\lambda \in \mathbb{N}$.
	\end{remark}

Given $s \in S$, we say that a representation $s= \sum_i s_i n_i$ is maximal if $\sum_i s_i=\ord_S(s)$.

\begin{lemma} \label{orders}
For a given integer  $t$ in $T=\du$ with maximal representation  $t=\sum_i r_i (2n_i) + \sum_j s_j(2e_j+b)$, the following statements hold:
	\begin{enumerate} 
	\item If $t$ is odd, then $\sum_j s_j=1$; 
	\item If $t=2s$ is even, then $\sum_j s_j =0$ and $\ord_T(t)=\ord_S(s)$.
	\end{enumerate}
\end{lemma}

\begin{proof}
	 Suppose that $\sum_j s_j \geq 2$, i.e. in the maximal representation of $t$ there are two elements of the form $2e_{j_1} + b$ and $2e_{j_2} + b$, not necessarily different. Then, it is possible to replace them with $2e_{j_1}$, $2e_{j_2}$ and $2b$ that are three elements of $T$. In this way we increase the summands in the representation and this is a contradiction, since it is maximal. 
	 Hence, if $t$ is odd, then $\sum_j s_j=1$, whereas if $t$ is even we have $\sum_j s_j=0$. In the latter case $s=\sum_i r_i n_i$ and, thus, $\ord_S(s) \geq \ord_T (t)$. Moreover, if $\sum_k p_k n_k$ is a maximal representation of $s$, then $t=\sum_k p_k (2n_k)$ and the thesis follows immediately. 	
\end{proof}

\begin{remark}\label{odd order}
	In the setting of the previous lemma, if $t$ is odd and  $t=\sum_i r_i (2n_i) + (2e+b)$ is a maximal representation, $e$ is necessarily a minimal generator of $E$, otherwise $e=e'+s$ and we increase the summands in the representation. Moreover, setting $s=\sum_i r_i n_i$, we have that $\ord_T(t)\geq \ord_S(s)+1$ and, conversely, if  $s=\sum_i r_i n_i$ is not a maximal representation, the same holds for
	the representation of $t$; hence, $\ord_T(t)= \ord_S(s)+1$.
	
	On the other hand, if we only assume that $t=2s+(2e+b)$, with $e$ minimal generator of $E$, we cannot conclude that $\ord_T(t)= \ord_S(s)+1$, since it could be $t=2s+(2e+b)=2s'+(2e'+b)$ with $\ord_S(s)<
	\ord_S(s')$.
\end{remark}

Let $f(S)=\max(\mathbb{Z} \setminus S)$. The {\it standard canonical ideal} of $S$ is defined as 
\[K(S)=\{x \in \mathbb{N} \mid f(S)-x \notin S\}.\]
It is characterized by the following duality property: $K(S)-(K(S)-F)=F$, for any relative ideal $F$ of $S$;
the same property holds for any shift $x+K(S)$. Starting by this fact, that is the numerical 
counterpart of the duality for canonical ideals in the one-dimensional rings, J\"ager proved in \cite[Satz 5]{J} that 
any fractional ideal of $R=k[\![S]\!]$, with valuation $x+K(S)$, is a canonical (fractional) ideal of $R$.
In particular, the monomial ideal corresponding to a proper ideal $E=K(S)+x \subseteq S$, with $x \in S$, provides a canonical ideal of $R$; consequently, $E$ is called a {\it canonical ideal} of $S$.
 A numerical semigroup is said to be {\it symmetric} if $S=K(S)$ and this notion is the corresponding one of the Gorenstein property in numerical semigroup theory; more precisely, $R$ is Gorenstein if and only if $S$ is symmetric.
We recall also that $S$ is symmetric if and only if $\delta_i + \delta_{m-i+1} = \delta_m$ for all $i=2, \dots, m-1$, see \cite[Proposition 4.10]{RG}.

\begin{proposition} \label{Mpure}
	Let $E$ be a canonical ideal of $S$. Then, $T=\du$ is $M$-pure if and only if $S$ is $M$-pure. Moreover, in this case $\ord_T(2\beta_m+b)=\ord_S(\delta_m)+1$.
\end{proposition}

\begin{proof}
	If $E=K(S)+x$, it is clear that $\Ap_{m}(E)=\Ap_{m}(K(S))+x$ and, thus, $\beta_1=x$ and $\beta_{m}=f(S)+m + x=\delta_{m}+x$.
	Furthermore, since $T$ is symmetric by \cite[Proposition 3.1]{DS}, it easily follows that $2 \delta_i + 2\beta_{m+1-i}+b=2 \beta_{m} + b$ for all $i=2, \dots, m$. 
	
	 Assume that $T$ is $M$-pure. Since $T$ is symmetric, this means that $\ord_T(2 \delta_i) + \ord_T(2\beta_{m+1-i}+b)=\ord_T(2 \beta_{m} + b)$ for all $i$ and, in particular, 
	$$
	\ord_T(2 \beta_m+b)=\ord_T(2 \delta_m)+\ord_T(2x+b)=\ord_S(\delta_m)+1,
	$$
since $2x+b$ is the smallest odd element of $T$. Let $\delta_i$ be a maximal element of $\Ap_m(S)$ with respect to $\leq_{M(S)}$ and let $\sum_j r_j (2n_j) +(2e+b)$ be a maximal representation of $2\beta_{m+1-i}+b$ in $T$. Suppose by contradiction that $\sum_j r_j >0$. Clearly, $\delta_i + \sum_j r_j n_j = \beta_m-e \in S$ and we claim that $\beta_m-e \in \Ap_m(S)$: in fact, $\beta_m-e-m=\delta_m-(e-x)-m=f(S)-k$ for some $k \in K(S)$ and, thus, it is not in $S$ by definition of $K(S)$. Moreover, since $2(\beta_m-e)+2e+b=2\beta_m+b$ and $2e+b$ is a minimal generator of $T$ which is $M$-pure, it follows that $\ord_T(2\beta_m+b)=\ord_S(\beta_m-e)+1$. Furthermore, we already know that 
	$$
	\ord_T(2 \beta_m+b)=\ord_T(2\delta_i)+\ord_T\left(\sum_j r_j (2n_j) + (2e+b)\right)= \ord_S (\delta_i)+\sum_j r_j +1.
	$$ 
	Hence, $\ord_S(\delta_i)+ \sum_j r_j=\ord_S(\beta_m-e)$ and this is a contradiction, since $\delta_i$ is maximal with respect to $\leq_{M(S)}$. 
	Consequently, we get $\sum_j r_j=0$ and, therefore,  
	$$
	\ord_S(\delta_m)+1=\ord_T(2\beta_m+b)=\ord_T(2\delta_i)+\ord_T(2e+b)=\ord_S(\delta_i)+1.
	$$
	This implies that $\ord_S(\delta_i)=\ord_S(\delta_m)$ and, hence, $S$ is $M$-pure.
	
	Conversely, assume that $S$ is $M$-pure. 
	First of all we claim that $\ord_T(2\beta_m+b)=\ord_S(\delta_m)+1$. Let $\sum_j r_j (2n_j) + (2e+b)$ be a maximal representation of $2\beta_m+b$ in $T$, again $e$ has to be a minimal generator of $S$ and it is easy to see that $\sum_j r_j \in \Ap_m(S)$. Therefore, 
	$$
	\ord_T(2\beta_m+b)=\sum_j r_j+1 = \ord_S \left(\sum_j r_j n_j \right)+1 \leq \ord_S(\delta_m)+1,
	$$
	where the inequality follows from the fact that $S$ is $M$-pure and $\delta_m$ is a maximal element of $S$. Moreover, we know that $2\delta_m + 2\beta_1+b = 2\beta_m + b$ and then $\ord_T(2\beta_m+b)\geq \ord_S(\delta_m)+1$; hence, we get the equality.
	
	It is enough to show that $2\beta_m+b$ is the only maximal elements of $\Ap_{2m}(T)$ with respect of $\leq_{M(T)}$. Since $T$ is symmetric, we know that $2 \delta_i + 2\beta_{m+1-i}+b=2 \beta_{m} + b$ and, then, it is sufficient to show that $2\delta_i \leq_{M(T)} 2 \beta_m + b$ for all $i=1, \dots, m$. Moreover, if $\delta_i$ is not maximal in $\Ap_m(S)$, it is clear that it is not maximal in $\Ap_{2m}(T)$ and, thus, it is enough to consider the elements $2 \delta_i$ with $\delta_i$ maximal in $\Ap_m(S)$. Finally, from $2 \delta_i + 2 \beta_{m+1-i}+b= 2\beta_m+b$ it follows that 
	$$
	\ord_S(\delta_m)+1=\ord_T (2\beta_m+b) \geq \ord_T(2\delta_i)+\ord_T(2\beta_{m+1-i}+b) \geq \ord_S(\delta_i)+1=\ord_S(\delta_m)+1
	$$
	for every maximal element $\delta_i$ in $\Ap_m(S)$. Hence, $\ord_{T}(2\beta_m+b) = \ord_T(2\delta_i)+\ord_T(2\beta_{m+1-i}+b)$ and $2\delta_i \leq_{M(T)} 2\beta_m+b$.
\end{proof}

\begin{example}
	If $E$ is not a canonical ideal, both the implications of the previous proposition do not hold. For instance $S=\langle 5,6,7 \rangle$ is $M$-pure, but $S \! \Join^5 \! (\{5,13\}+S)= \langle 10,12,14,15,31 \rangle$ is not $M$-pure. On the other hand, $S=\langle 6,7,22 \rangle$ is not $M$-pure, but $S \! \Join^7 \! (\{6,21\}+S)=\langle 12,14,19,44,49 \rangle$ is $M$-pure.
\end{example}

	\begin{theorem} \label{Gorenstein}
		Let $D=k[\![\du]\!]$ and let $\fn$ be its maximal ideal. Then $\gr_\fn(D)$ is Gorenstein if and only if $S$ is $M$-pure, $E$ is a canonical ideal and $\gr_\fm(R)$ is Cohen-Macaulay. 
	\end{theorem}
	
	\begin{proof}
	Recall that $\gr_\fn(D)$ is Gorenstein if and only if it is Cohen-Macaulay and $\du$ is $M$-pure and symmetric \cite[Theorem 3.14]{B}. 
	
	Assume that $\gr_\fn(D)$ is Gorenstein. By \cite[Theorem 3.14]{B} this is equivalent to say 
	that it is Cohen-Macaulay and $\du$ is $M$-pure and symmetric. Since $\du$ is symmetric,
	$E$ has to be a canonical ideal of $S$ (\cite[Proposition 3.1]{DS}). So we can apply the previous proposition to obtain that $S$ is $M$-pure. Finally, by Corollary \ref{CM-Gor} we get that also 
	$\gr_\fm(R)$ is Cohen-Macaulay.
	
	Conversely, assume that $S$ is $M$-pure, $E$ is a canonical ideal and $\gr_\fm(R)$ is Cohen-Macaulay.
	Hence Proposition \ref{Mpure} and \cite[Proposition 3.1]{DS} imply that $\du$ is $M$-pure
	and symmetric. Therefore we have only to show that $\gr_\fn(D)$ is Cohen-Macaulay.
	By Remark \ref{Bryant} it is enough to show that $\ord_T(2\beta_m+b+ 2\lambda m)=
	\ord_T(2\beta_m+b)+\lambda$ for all $\lambda \in \mathbb{N}$.
	Since $\gr_\fm(R)$ is Cohen-Macaulay and $\ord_T(2\beta_m+b)=\ord_S(\delta_m)+1$ (again by Proposition \ref{Mpure}),
	it is enough to show that $\ord_T(2\beta_m+b+ 2\lambda m) = \ord_S(\delta_m+\lambda m)+1$ for all $\lambda \in \mathbb{N}$. If $E=x+K(S)$, then $\beta_m=\delta_m+x$ and
		\[\ord_T(2\beta_m+b+ 2\lambda m)=\ord_T(2\delta_m+ 2\lambda m+2x+b)\geq \ord_S(\delta_m+\lambda m)+1.
		\]
		Let $2\beta_m+b + 2 \lambda m= \sum r_i (2n_i) + (2e+b)$ be a maximal representation, where $e$ is a minimal generator of $E$. Clearly, there exists $\gamma \in \mathbb{N}$ such that $\sum r_i n_i - \gamma m \in \Ap_m(S)$. If $\gamma > \lambda$, then 
		$2\beta_m+b > 2 (\sum r_i n_i - \gamma m) + 2e+b \in T$ and, since they are congruent module $2m$, this yields a contradiction; hence, $\gamma \leq \lambda$.
		Moreover, $\ord_S(\delta_m)$ is the maximum order of the elements of $\Ap_m(S)$, so using Corollary \ref{CM2}, it follows that
		\begin{equation*}
		\begin{split}
		\ord_T(2\beta_m+b+ 2\lambda m)&=\ord_S\left(\sum r_i n_i\right)+1=\ord_S\left(\sum r_i n_i-\gamma m\right) + \gamma +1 \leq \\
		&\leq  \ord_S(\delta_m)+\lambda +1 = \ord_S(\delta_m+\lambda m)+1. \qedhere
		\end{split}
		\end{equation*}
		
	\end{proof}

	
	
	The next result is a consequence of Corollary \ref{CM-Gor}, since a Gorenstein ring is Cohen-Macaulay.
		
	\begin{corollary}
If $E$ is a canonical ideal of an $M$-pure numerical semigroup $S$ and $\gr_\fm(R)$ is Cohen-Macaulay, then $\gr_{\fm}(I)$ is Cohen-Macaulay. 
	\end{corollary}

Recalling that a principal ideal is canonical if and only if the ring is Gorenstein, we get the following corollary.

\begin{corollary}
	$\gr_\fm(R)$ is Gorenstein if and only if $\gr_\fn(k[\![\du]\!])$ is Gorenstein for every principal ideal $E$ of $S$. 
\end{corollary}

\begin{example}
	Consider the numerical semigroup $S=\langle 10,11,12,13 \rangle$. This is $M$-pure and its associated graded ring is Cohen-Macaulay, but $S$ is not symmetric and, therefore, $\gr_{\fn}(R)$ is not Gorenstein. On the other hand, if $E$ is a canonical ideal of $S$, $\gr_\fn(k[\![\du]\!])$ is Gorenstein for all odd $b \in S$ by Theorem \ref{Gorenstein}. For instance, $E=\{10,11,12\}+S$ is a canonical ideal and in this case $S \! \Join^{11} \! E=\langle 20,22,24,26,31,33,35 \rangle$.  
\end{example}

As in the example above, there are several cases in which $\gr_\fm(R)$ is not Gorenstein, but $\gr_\fn(k[\![\du]\!])$ is.

	\begin{corollary} {\rm (\cite[Proposition 3.22]{B})}
		Let $E$ be a canonical ideal of $S$. If $S$ satisfies one of the following conditions, then $\gr_\fn(k[\![\du]\!])$ is Gorenstein.
		\begin{enumerate}
			\item $S$ has embedding dimension $2$.
			\item $S$ has maximal embedding dimension, i.e. $\nu=m$.
			\item $S$ is symmetric and $\nu=m-1$.
			\item $S$ is generated by an arithmetic sequence.
			\item $m \leq 4$, except the case $S=\langle 4, n_2, n_3 \rangle$ such that $S$ is not symmetric.
		\end{enumerate}
	\end{corollary}
	
	For instance, if $S$ has maximal embedding dimension and $m \neq 2$, then it is not symmetric and so $\gr_\fm(R)$ is not Gorenstein, but $\gr_\fn(k[\![\du]\!])$ is Gorenstein for every odd $b \in S$ and every canonical ideal $E$ of $S$.
	
Let $\omega_R$ be a canonical module of $R$. By Corollary \ref{idealization}, it follows that $\gr_\fn(k[\![\du]\!])$ is Gorenstein if and only if $\gr_\fm(\omega_R)$ is a canonical module of $\gr_\fm(R)$. In general, in \cite[Theorem 3.5]{Oo} A. Ooishi proves that, if $R$ is local, the associated graded module of $\omega_R$ is a canonical module of $\gr_\fm(R)$ exactly when $\gr_\fm(R)$ and $\gr_\fm(\omega_R)$ are Cohen-Macaulay and $H(\gr_\fm(\omega_R),t)=(-1)^d t^a H(\gr_\fm(R),t^{-1})$, where $H(M,t)$ denotes the Hilbert series of a module $M$, $d=\dim R$ and $a$ is the $a$-invariant of $\gr_\fm(R)$.

In the light of Theorem \ref{Gorenstein}, in the case of numerical semigroup rings, this technical condition can be replaced by the  $M$-pureness for $S$:

\begin{corollary} \label{canonical module}
	Let $R=k[\![S]\!]$ be a numerical semigroup ring with canonical module $\omega_R$. Then, $\gr_\fm(\omega_R)$ is a canonical module of $\gr_\fm(R)$ if and only if $\gr_\fm(R)$ is Cohen-Macaulay and $S$ is $M$-pure.
\end{corollary}

\begin{proof}
	We can assume that $\gr_\fm(R)$ is Cohen-Macaulay. By Corollary \ref{CM-Gor}, $\gr_\fm(\omega_R)$ is a canonical module of $\gr_\fm(R)$ if and only if $\gr_\fn(k[\![S\!\Join^b \!E]\!])$ is Gorenstein, where $E$ is a canonical ideal of $S$ and $b$ an odd element of $S$. Then, the thesis follows from Theorem \ref{Gorenstein}.
\end{proof}

\subsection{Complete intersections}	In this subsection we focus to the complete intersection property for $\gr_\fn(k[\![\du]\!])$. We will make use of a characterization proved in \cite{DMS}: $\gr_\fm(R)$ is a complete intersection if and only if it is Cohen-Macaulay and $S$ has $\gamma$-rectangular Ap\'ery set. To explain this result we need some notation. For every $i=2, \dots, \nu$ we define
\begin{align*}
\beta_S(n_i) =  \max \{ h \in \mathbb{N} \mid & \, hn_i \in \Ap_{m}(S) {\rm \ and \ } \ord(hn_i)=h \}; \\
\gamma_S(n_i)=  \max \{ h \in \mathbb{N} \mid &\, hn_i \in \Ap_{m}(S), \, \ord(hn_i)=h {\rm \ and \ } \\
& \, hn_i {\rm \ has \ a \ unique \ maximal \ representation }\}.
\end{align*}
(we write $\beta_S(n_i)$ instead of $\beta_i$ as in \cite{DMS}, since in that paper the generators are 
listed in increasing order, while the ordering of the generators of $T=\du$ could not be clear).
Let $B(S)=\{\sum_2^\nu \lambda_i n_i \mid 0 \leq \lambda_i \leq \beta_S(n_i)\}$ and $\Gamma(S)=\{\sum_2^\nu \lambda_i n_i \mid 0 \leq \lambda_i \leq \gamma_S(n_i)\}$. In \cite[Corollary 2.7]{DMS} it is proved that $\Ap_{m}(S) \subseteq \Gamma (S) \subseteq B(S)$ and we will say that a numerical semigroup has $\beta$- or $\gamma$-rectangular Ap\'ery set if $\Ap_{m}(S)=B(S)$ or $\Ap_{m}(S)=\Gamma(S)$, respectively. 

\begin{proposition} \label{complete intersection} 
	Let $T=\du$ and $D=k[\![\du]\!]$.
	\begin{enumerate}
	\item $T$ has $\gamma$-rectangular (resp. $\beta$-rectangular) Ap\'ery set if and only if $S$ has $\gamma$-rectangular (resp. $\beta$-rectangular) Ap\'ery set and $E$ is principal; 
	\item $\gr_\fn(D)$ is complete intersection if and only if $\gr_\fm(k[\![S]\!])$ is complete intersection and $E$ is principal.
	\end{enumerate}	
\end{proposition}

\begin{proof} 
	(1) It easily follows from Lemma \ref{orders} that $\beta_T(2n_i)=\beta_S(n_i)$ and $\gamma_T(2n_i)=\gamma_S(n_i)$ for all $i=2, \dots, \nu$. On the other hand, if $2e_i+b$ is a generator of $T$, then $\beta_T(2e_i+b)=\gamma_T(2e_i+b)=1$, since $2(2e_i+b)=2e_i + 2e_i + 2b$ has order greater than $2$. This implies that the number of the even elements in $\Gamma(T)$ is greater than $|\Gamma(S)|$ and equal to it if and only if $E$ is principal; in the same way, the number of the odd elements in $\Gamma(T)$ is greater than $|\Gamma(S)|$ and they are equal if and only if $E$ is principal. Therefore,
	$$
	|\Ap_{2m}(T)|=2|\Ap_{m}(S)| \leq 2 |\Gamma(S)| \leq |\Gamma(T)|
	$$  
	and, hence, the Ap\'ery set of $T$ is $\gamma$-rectangular if and only if all the inequalities above are equalities, i.e. $S$ has $\gamma$-rectangular Ap\'ery set and $E$ is principal.
	As for the $\beta$-rectangularity, it is enough to apply the same argument. \\ 
	(2) Recall that the associated graded ring of a numerical semigroup ring is complete intersection if and only if it is Cohen-Macaulay and the Ap\'ery set of the semigroup is $\gamma$-rectangular, see \cite[Theorem 3.6]{DMS}. Therefore, the thesis follows from the first part of this proposition and by Corollary \ref{principal}.
\end{proof}

\section{Homogeneous property}

Consider the surjective homomorphism $\varphi: k[\![x_1, \dots, x_\nu]\!] \rightarrow k[\![S]\!]$ that associates the minimal generators $n_i$ with every $x_i$. It induces an isomorphism between $k[\![S]\!]$ and $k[[x_1, \dots ,x_\nu]/I_S$ for some ideal $I_S$, called the defining ideal of $S$. In this section we are interested in comparing the Betti numbers of $k[\![S]\!]$ and $\gr_\fm(k[\![S]\!])$ as $k[\![x_1, \dots, x_\nu]\!]$-modules.
It is well-known that the $i$-th  Betti number $\beta_i(k[\![S]\!])$ of $k[\![S]\!]$ is less than or equal to $\beta_i(\gr_\fm(k[\![S]\!]))$ for all $i$; if all the equalities hold, i.e. $k[\![S]\!]$ and $\gr_\fm(k[\![S]\!])$ have the same Betti numbers, $k[\![S]\!]$ (and $S$) is said to be of homogeneous type \cite{HRV}. In \cite{JZ} it is introduced the notion of homogeneous semigroup in order to find numerical conditions assuring that $k[\![S]\!]$ is of homogeneous type. In fact, it is proved that if $S$ is homogeneous and $\gr_\fm(k[\![S]\!])$ is Cohen-Macaulay, then $k[\![S]\!]$ is of homogeneous type.

In this section we characterize when the numerical duplication $\du$ is 
homogeneous in terms of $S$, $E$ and $b$.
Let $\mathcal{L}(z)$ be the set of the lenghts of the representations of an integer $z$. 
In the following we recall the definition of homogeneous numerical semigroup introduced in \cite{JZ} 
and we generalize it to ideals.

\begin{definition} Let $E=\{e_1, \dots, e_r\}+S$ be an ideal of $S$ and let $\Ap_{m}(E)=\{\beta_1, \beta_2, \dots, \beta_{m} \}$ be the  Ap\'ery set of $E$ with respect to the multiplicity.
\begin{enumerate}
	\item An integer $z$ is called homogeneous for $S$ if either $z \notin S$ or $\mathcal{L}(z)$ is a singleton. 
	\item The semigroup $S$ is said to be homogeneous if every element of $\Ap_{m}(S)$ is homogeneous. 
	\item The ideal $E$ is called homogeneous if $\beta_i - e_j$ is homogeneous for all $i,j$ 
	and all the non-empty sets among 
	$\mathcal{L}(\beta_i - e_1)$, $\mathcal{L}(\beta_i - e_2)$, $\dots$, $\mathcal{L}(\beta_i - e_r)$ 
	are equal for all $i$.
\end{enumerate}
\end{definition}

\begin{examples} \label{No arrows}
{\bf 1.} Let $S_1=\langle 4,5 \rangle$ be a numerical semigroup and consider 
the ideal $E_1=\{5,8\}+S$ of $S_1$. The Ap\'ery set of $E_1$ is $\Ap_4(E_1)=\{5,8,10,15\}$ and it is 
easy to see that $E_1$ is homogeneous. \\
{\bf 2.} Consider again $S_1= \langle 4,5 \rangle$ and $E_2=\{5\}+S$. 
We have $\Ap_4(E_2)=\{5,10,15,20\}$ and one can easily check that it is homogeneous. \\ 
{\bf 3.} Let $S_2= \langle 6,7,9,11 \rangle$ and consider its ideal $E_3=\{7,11,12\}+S$. 
In this case we have $\Ap_6(E_3)=\{7,11,12,14,16,21\}$ and  $E_3$ is not homogeneous, 
since $9=21-12$ and $14=21-7$ have different order.
\end{examples}

If every element of $\Ap_m(E)$ is homogeneous and its generators have the same order as elements of $S$, 
it is easy to see that $E$ is homogeneous. The converse is not true as the first two examples show. 

\begin{lemma} \label{3E+b}
Let $\Ap_{m}(S)=\{\delta_1, \dots, \delta_{m}\}$ and $\Ap_{m}(E)=\{\beta_1, \dots, \beta_{m} \}$.
If $\delta_i \notin 2E+b$ for every $i$, then $\beta_j \notin 3E+b$ for every $j$.
\end{lemma}

\begin{proof}
Assume by contradiction that $\beta_j=a_1+a_2+a_3+b$, for some $j$ with $a_i \in E$. 
By hypothesis $a_2+a_3+b \notin \Ap_m(S)$, then $a_2+a_3+b-m \in S$ and, thus, 
$\beta_j-m=a_1+(a_2+a_3+b-m) \in E$, that is a contradiction. 
\end{proof}

The next theorem shows the importance of the homogeneous property for semigroup ideals.

\begin{theorem} \label{homogeneous duplication}
	Let $\Ap_{m}(S)=\{\delta_1, \dots, \delta_{m}\}$ and 
	$\Ap_{m}(E)=\{\beta_1, \dots, \beta_{m} \}$. 
	Then, $\du$ is homogeneous if and only if $S$ is homogeneous, 
	$E$ is homogeneous and $\delta_i \notin 2 E +b$ for all $i$. 
\end{theorem}

\begin{proof}
 Recall that 
	$\Ap_{2m}(\du)=\{2 \delta_1, \dots, 2\delta_{m}, 2\beta_1 +b, \dots, 2 \beta_{m} + b\}$.
	Assume first that $\du$ is homogeneous. 
	If $\sum_{j}r_j n_j$ and $\sum_k s_k n_k$ are two representations of 
	$\delta_i$, then $\sum_{j}r_j (2n_j)$ and $\sum_k s_k (2n_k)$ are two representations of $2 \delta_i$; 
	therefore, $\sum_{j}r_j =\sum_k s_k$ since $\du$ is homogeneous, and so $S$ is homogeneous as well.
	
	Let $\sum_{k}r_k n_k$ and $\sum_l s_l n_l$ be two representations of 
	$\beta_i - e_j$. Then, $2 \beta_i +b= 2(\beta_i-e_j)+2e_j+b$ has the two representations 
	$\sum_{k}r_k (2n_k) + 2e_j+b$ and $\sum_l s_l (2n_l)+2e_j+b$. 
	Since $2e_j+b$ is a minimal generator of $\du$, 
	the lengths of the representations above are $1+\sum_{k}r_k$ 
	and $1+\sum_l s_l$; thus, $\sum_{k}r_k=\sum_l s_l$, because $\du$ is homogeneous. 
	Moreover, if $\beta_i-e_{j_1}=\sum_k r_k n_k$ and $\beta_i - e_{j_2}= \sum_l s_ln_l$, 
	as above we get $2\beta_i +b =\sum_k r_k (2n_k)+2e_{j_1}+b=\sum_l s_l(2n_l)+2e_{j_2}+b$ 
	and again $\sum_k r_k=\sum_l s_l$; hence, $E$ is homogeneous.
	
	Now assume by contradiction that $\delta_i \in 2E+b$ for some $i$, 
	i.e. $\delta_i=\sum_j r_j n_j + \sum_k s_k e_k +b$ with $\sum_k s_k \geq 2$; 
	since $e_k \in S$, we can assume that $\sum_k s_k = 2$ and, thus, 
	$\delta_i= \sum r_j n_j + e_{k_1}+e_{k_2}+b$. 
	Therefore, $2 \delta_i= \sum_j r_j(2n_j) + 2e_{k_1}+b + 2e_{k_2}+b$ and, consequently, 
	$2 \delta_i$ has a representation in $\du$ of length $\sum_j r_j +2$; on the other hand, 
	$\ord_{S \Join^{b} E} 2\delta_i = \ord_S \delta_i \geq \sum_j r_j +3$ and this yields a contradiction.
	
	
	Conversely, assume that $S$ is homogeneous, $E$ is homogeneous and 
	$\delta_i \notin 2E+b$ for every $i$. 
	Let $2 \delta_i = \sum_j r_j (2n_j)+ \sum_k s_k (2e_k+b)$. 
	Clearly, $\sum_k s_k$ has to be even, since $b$ is odd. If $\sum_k s_k >0$, it follows that 
	$$\delta_i= \sum_j r_j n_j + \sum_k s_k e_k + \frac{\sum_k s_k}{2}b$$ 
	and, since $\sum_k s_k \geq 2$, we get $\delta_i \in 2E+b$; contradiction. 
	Hence, $2 \delta_i= \sum_j r_j(2n_j)$. Let $\sum_k r_k' (2n_k)$ be another representation of $2\delta_i$. 
	Then, $\delta_i = \sum_j r_j n_j=\sum_k r_k' n_k$ and, since $S$ is homogeneous, 
	it follows that $\sum_j r_j = \sum_k r_k'$.
	
	Now let $2 \beta_i+b = \sum_j r_j (2n_j) + \sum_k s_k (2e_k+b)$. Here $\sum_k s_k$ is odd and 
	suppose by contradiction that $\sum_k s_k \geq 3$. Then 
	$$
	\beta_i= \sum_j r_j n_j + \sum_k s_k e_k + \frac{(\sum_k s_k)-1}{2}b
	$$
	and, thus, $\beta_i \in 3E+b$, that contradicts Lemma \ref{3E+b}.
	Therefore, $2 \beta_i+b = \sum_j r_j (2n_j) + 2e_{k_{1}} + b$. 
	Consequently, if $\sum_l r'_l (2n_l) + 2e_{k_2}+b$ is another representation of $2 \beta_i +b$, 
	it immediately follows that $\sum_j r_j = \sum_l r'_l$, since $E$ is homogeneous. 
\end{proof}

The conditions $\delta_i \notin 2 E +b$ are true for $b \gg 0$, 
consequently if $S \! \Join^{\overline b} \! E$ is homogeneous for some $\overline{b}$, 
then $\du$ is always homogeneous for $b \gg 0$.

\begin{example}
	The ideals $E_1=\{5,8\}+S$ and $E_2=\{5\}+S$ of Examples \ref{No arrows} 
	are both homogeneous and $S_1=\langle 4,5 \rangle$ is homogeneous. 
	The Ap\'ery set of $S_1$  is $\Ap_4(S)=\{0,5,10,15\}$ and, 
	since $15$ is both in $2E_1 + 5$ and $2E_2+5$, 
	it follows that $S_1 \! \Join^5 \! E_1=\langle 8,10,15,21 \rangle$ 
	and $S_1 \! \Join^5 \! E_2=\langle 8,10,15 \rangle$ are not homogeneous. 
	On the other hand, the smallest element of $2E_1+b$ and $2E_2+b$ is $10+b$, 
	thus, if $b>5$, the elements of $\Ap_4(S)$ are neither in $2E_1+b$ nor in $2E_2+b$; 
	it follows that both $S_1 \! \Join^b \! E_1$ and $S_1 \! \Join^b \! E_2$ are 
	homogeneous for any odd $b \in S$ except $5$. 
	For instance $S_1 \! \Join^9 \! E_1=\langle 8,10,23,29 \rangle$ and 
	$S_1 \! \Join^9 \! E_2= \langle 8,10,23 \rangle$ are homogeneous.
\end{example}

	On the one hand Theorem \ref{homogeneous duplication} 
	allows to construct homogeneous numerical semigroups (and, consequently, of homogeneous type), 
	but on the other hand it can be also  used to construct semigroups that are not homogeneous. 
	The following result gives an answer to \cite[Question 4.22]{JZ}.
	
	\begin{proposition}
Let $S$ be of homogeneous type. If there exists $s \in M$ 
and an odd integer $b \in S$ such that $2s+b \in \Ap_{m}(S)$, then $S \! \Join^{b} \! E$ 
is of homogeneous type and not homogeneous, where $E=\{s\}+S$. 
Moreover, $\gr_\fn(k[\![\du]\!])$ is not complete intersection, 
provided that $\gr_\fm(k[\![S]\!])$ is not complete intersection.
	\end{proposition}
	\begin{proof}
		Since $E$ is principal, the numerical duplication coincides with a particular
		case of the so-called simple gluing and, then, \cite[Theorem 5.2]{St} 
		implies that  $\du$ is of homogeneous type, but it is not homogeneous by 
		Theorem \ref{homogeneous duplication}.
	In the case that $\gr_\fm(k[\![S]\!])$ is not complete intersection, 
	the result follows by Proposition \ref{complete intersection}.
	\end{proof}

\begin{examples} \label{example}
	{\bf 1.} Consider $S=\langle 6,7,10 \rangle$. 
	Its defining ideal is $(x_1x_2^2 - x_3^2, x_1^4 - x_2^2 x_3, x_2^4-x_1^3 x_3)$
	and the set of the minimal generators is also a standard basis. 
	Since $S$ has embedding dimension three, it is a numerical semigroup 
	of homogeous type by \cite[Theorem 4.5]{JZ}. 
	Its Ap\'ery set with respect to $6$ is $\Ap_6(S)=\{0,7,10,14,17,21\}$ and, 
	then, $3 \cdot 7 \in \Ap_6(S)$. Therefore, 
	the previous proposition implies that $T=S\! \Join^{7} \! \langle 7 \rangle=\langle 12,14,20,21 \rangle$ 
	is of homogeneous type, but $T$ is not homogeneous and its associated graded ring is not complete intersection. 
	In fact, $63=3 \cdot 21=3 \cdot 14+21 \in \Ap_{12}(T)$ and, according to Macaulay2 \cite{M2}, 
	the minimal free resolutions of $k[\![T]\!]$ and its associated ring are
	$$
	0 \rightarrow A^2 \rightarrow A^5 \rightarrow A^4 \rightarrow N \rightarrow 0, 
	$$
	where $A=k[\![x_1, \dots, x_4]\!]$ and $N$ is either $k[\![T]\!]$ or $\gr_\fn(k[\![T]\!])$. \\[1mm]
	{\bf 2.} In the previous example we use \cite[Theorem 4.5]{JZ}: if $S$ has 
	embedding dimension three and $\beta_1(k[\![S]\!])=\beta_1(\gr_\fm(k[\![S]\!]))$, 
	then $\beta_i(k[\![S]\!])=\beta_i(\gr_\fm(k[\![S]\!]))$ for all $i$, i.e. $S$ is of 
	homogeneous type. This is not true if $S$ has $4$ minimal generators; for instance 
	if $S=\langle 6,7,8,17 \rangle$, its defining ideal is generated by the set
	$\{x_2^2-x_1x_3,x_2x_3^2-x_1x_4,x_3^3-x_2x_4,x_1^4-x_2x_4,x_1^3x_2-x_3x_4, x_1^3x_3^2-x_4^2\}$ 
	that is also a standard basis, but $\gr_\fm(k[\![S]\!])$ is not Cohen-Macaulay.
\end{examples}


\begin{thebibliography}{99}
	
	\bibitem{AW} D.D. Anderson,  M. Winders, {\em Idealization of a module}, J. Commut. Algebra {\bf 1} (2009), no. 1, 3--56.
	
	
\bibitem{BDS} V. Barucci, M. D'Anna, F. Strazzanti, {\em A family of quotients of the Rees algebra}, Commun. Algebra {\bf 43} (2015), no. 1, 130--142.
	
	\bibitem{BDS2} V. Barucci, M. D'Anna, F. Strazzanti, {\em Families of Gorenstein and almost Gorenstein rings}, Ark. Mat. {\bf 54} (2016), no. 2, 321--338.
	
	\bibitem{BDF} V. Barucci, D.E. Dobbs, M. Fontana, Maximality properties in numerical semigroups and applications to one-dimensional analytically irreducible local domain, Mem. Amer. Math. Soc. {\bf 125}, no. 598, 1997.
	
	\bibitem{BF} V. Barucci, R. Fr\"{o}berg, {\em Associated graded rings of one dimensional analytically irreducible rings}, J. Algebra {\bf 304} (2006) n.1, 349--358.	
	
	\bibitem{B} L. Bryant, {\em Goto numbers of a numerical semigroup ring and the Gorensteiness of associated graded rings}, Commun. Algebra {\bf 38} (2010), no. 6, 2092--2128.
	
	\bibitem{D} M. D'Anna, {\em A construction of Gorenstein rings}, J. Algebra {\bf 306} (2006), no. 2, 507--519.	
	
	\bibitem{DF}
M. D'Anna, M. Fontana, {\em An amalgamated duplication of a ring along an ideal: basic properties}, J. Algebra Appl. {\bf 6} (2007), no. 3, 443--459.	
	
	\bibitem{DMS} M. D'Anna, V. Micale, A. Sammartano, {\em When the associated graded ring of a semigroup ring is Complete Intersection}, J. Pure Appl. Algebra {\bf 217} (2013), no. 6, 1007--1017.
	
	\bibitem{DS} M. D'Anna, F. Strazzanti, {\em The numerical duplication of a numerical semigroup}, Semigroup Forum {\bf 87} (2013), no. 1, 149--160.
	
	\bibitem{DGM} M. Delgado, P.A. Garc\'ia-S\'anchez, J. Morais, {\em ``NumericalSgps'' -- a GAP package}, Version 1.1.5 (2017) \href{http://www.gap-system.org/Packages/numericalsgps.html}{http://www.gap-system.org/Packages/numericalsgps.html}.	
	
	
	
	\bibitem{FGR} R.M. Fossum, 	P.A. Griffith, I. Reiten, Trivial extensions of abelian categories. Homological algebra of trivial extensions of abelian categories with applications to ring theory. Lecture Notes in Mathematics, Vol. 456. Springer-Verlag, Berlin-New York, 1975. 
	
	\bibitem{GAP} The GAP Group, {\em GAP -- Groups, Algorithms, and Programming}, Version 4.8.4 (2016). \href{http://www.gap-system.org}{http://www.gap-system.org}.
	
	\bibitem{G}
	A. Garc\'ia, {\em Cohen-Macaulayness of the associated graded of a semigroup ring}, Commun. Algebra {\bf 10} (1982), 393--415. 
	
	\bibitem{J} J. J\"ager, {\em L\"angenberechnung und kanonische ideale
		in eindimensionalen ringen}, Arch. Math. {\bf 29} (1997),
	504--512.
	
	\bibitem{M2} D.R. Grayson, M.E. Stillman, {\em Macaulay2, a software system for research in Algebraic Geometry}, available at \href{http://www.math.uiuc.edu/Macaulay2/}{http://www.math.uiuc.edu/Macaulay2/}.	
	
	
	\bibitem{HRV} J. Herzog, M.E. Rossi, G. Valla, {\em On the depth of the symmetric algebra}, Trans. Amer. Math. Soc. {296} (1986), no. 2, 577--606.	
	
	\bibitem{Huang-2015} 
	I-C. Huang, {\em Residual complex on the tangent cone of a numerical semigroup ring}, Acta Math. Vietnam. {\bf 40} (2015), no. 1, 149--160. 
	
	
	\bibitem{JZ} R. Jafari, S. Zarzuela Armengou, {\em Homogeneous numerical semigroups}, \href{https://arxiv.org/abs/1603.01078v1}{arXiv:1603.01078v1} (2016).
	
	\bibitem{L} J. Lipman, {\em Stable ideals and Arf rings}, Amer. J. Math. {\bf 93} (1971), 649--685.
	
	
	\bibitem{OST} A. Oneto, F. Strazzanti, G. Tamone, {\em One-dimensional Gorenstein local rings with decreasing Hilbert function}, J. Algebra {\bf 489} (2017), 91--114. 	

	\bibitem{Oo} A. Ooishi, {\em On the associated graded modules of canonical modules}, J. Algebra {\bf 141} (1991), 143--157.

		
	\bibitem{RG} J.C. Rosales, P.A. Garc\'ia-S\'anchez, Numerical Semigroups, Springer Developements in Mathematics, Vol 20, 2009.

	\bibitem{St} D.I. Stamate, {\em Betti numbers for numerical semigroup rings}, to appear in Multigraded Algebra and Applications (V. Ene, E. Miller, Eds.) Springer Proceedings in Mathematics \& Statistics, \href{https://arxiv.org/abs/1801.00153}{arXiv:1801.00153v1}.

	
	
\end{thebibliography}
\end{document}